\documentclass[11pt]{article}
\usepackage[margin=1in]{geometry}
\usepackage{amsmath,amsthm,amssymb,mathrsfs,graphicx}
\usepackage{amsfonts}
\usepackage{enumitem}
\usepackage{hyperref}
\hypersetup{colorlinks,  linkcolor={blue}, citecolor={blue}, urlcolor={blue}}

\allowdisplaybreaks[4] 
\numberwithin{equation}{section}

\usepackage{authblk}

\newtheorem{theorem}{Theorem}[section]

\newtheorem{lemma}[theorem]{Lemma}
\newtheorem{corollary}[theorem]{Corollary}

\newtheorem{remark}[theorem]{Remark}

\newtheorem{claim}[theorem]{Claim}
\newtheorem{conjecture}[theorem]{Conjecture}

\def\RR{\mathbb{R}}

\def\A{\mathcal{A}}
\def\B{\mathcal{B}}
\def\F{\mathcal{F}}
\def\G{\mathcal{G}}

\def\S{\mathcal{S}}

\def\v{\mathbf{v}}

\def\f{\mathbf{f}}

\def\1{\mathbf{1}}

\begin{document}

\title{An Erd\H{o}s--Ko--Rado theorem for cross-intersecting families in the Euclidean inner product}

\author{
Jiang-Chao Wan\footnote{School of Mathematics and Statistics, Hefei University, Hefei 230601, Anhui, P.R. China.
E-mail: \texttt{wanjc@stu.ahu.edu.cn}. Supported by Talent Research Fund Project of Hefei University (25RC02).}, ~~~
Yi Wang\footnote{School of Mathematical Sciences,
Anhui University, Hefei 230601, P.R. China.
E-mail: \texttt{wangy@ahu.edu.cn}.
Supported by National Natural Science Foundation of China (12571360, 12331012),
Excellent University Research and Innovation Team in Anhui Province (2024AH010002, 2025AHGXZK10041).}
}


\date{\today}

\maketitle

\begin{abstract}
Let $\binom{[n]}{k}$ be the set of all $k$-element subsets of the set $\{1,\ldots,n\}$ and let $\A,\B \subseteq \binom{[n]}{k}$ be two cross-intersecting families, that is, $A\cap B\neq \emptyset$ for any $A\in \A$ and $B\in \B$.
The classical cross-intersecting version of the Erd\H{o}s--Ko--Rado theorem,
due to Pyber~\cite{Pyber} and Matsumoto--Tokushige~\cite{Matsumoto},
states that if $n\geq 2k$, then $|\A||\B|\leq \binom{n-1}{k-1}^2,$
where the equality holds for $n>2k$ if and only if $\A=\B$ is a star.
In the present paper, we first give a stability result of this theorem by using Filmus's FKN theorem on the slice~\cite{Filmus} and linear algebra method as follows:
There exists a constant $C>1$ such that if $n\geq 2.07k$ and $|\A||\B|\geq (1-\epsilon)\binom{n-1}{k-1}^2$,
where $\epsilon\leq \frac{k^2}{C^2 n ^2 }$,
then there is a star $\mathcal{S}$ such that
$|\mathcal{S} \Delta \A|\leq C \epsilon \binom{n}{k}$ and
$|\mathcal{S} \Delta \B|\leq C \epsilon \binom{n}{k}.$
Moreover, based on this stability result and the eigenvalues of the matrices of the Johnson scheme,
we present an Erd\H{o}s--Ko--Rado theorem for cross-intersecting families in the Euclidean inner product showing that if $n\geq 2k$ and $k\geq d \geq 0$, then
$$\big\langle\v_d(\A),\v_d(\B)\big\rangle
\leq \frac{\binom{k}{d}\binom{k-1}{d}}{\binom{n-1}{d}}\binom{n-1}{k-1}^2
+\binom{k-1}{d-1} \binom{n-d-1}{k-d}\binom{n-1}{k-1},$$
together with uniqueness and a corresponding stability result, where $\mathbf{v}_d(\A) \in \RR^{\binom{[n]}{d}}$ is the $d$-degree vector of $\A$ whose $U$-entry is the number of members in $\A$ containing $U$.
To achieve this, we develop a cross version of Bey's inequality~\cite{Bey} and the stability result of this cross version,
which allows us to provide a $t$-cluster generalization
of the Erd\H{o}s--Ko--Rado theorem in $\ell_p$-norm with $p\geq2$.
\end{abstract}

\section{Introduction}

Several fundamental problems in extremal combinatorics
concern determining the maximum size of a discrete structure
subject to certain restrictions.
One typical problem is to find the largest cardinality of an intersecting family of sets.
Let $\binom{[n]}{k}$ denote the set of all $k$-element subsets of the set
$\{1,\ldots,n\}=:[n]$.
A family $\A \subseteq \binom{[n]}{k}$ is called {\em intersecting} if
$S\cap T\neq \emptyset$ for any $S,T\in \A$.
The celebrated theorem of Erd\H{o}s, Ko, and Rado~\cite{Erdos} states that
if $n\geq 2k$ and $\A \subseteq \binom{[n]}{k}$ is an intersecting family,
then
$$|\A| \leq \binom{n-1}{k-1},$$
where the equality holds for $n>2k$ if and only if $\A$ is a star, i.e., the family $\S_i=\{S\in \binom{[n]}{k}: i\in S\}$ for a fixed $i\in[n]$.
The Erd\H{o}s--Ko--Rado theorem is widely regarded as a central result in extremal combinatorics.
Over the years, numerous extensions and applications of this theorem have been proved, we refer the interested reader to~\cite{Deza,FTsurvey} for surveys and~\cite{Frankl} for several different proofs.

Pyber~\cite{Pyber} presented a natural cross-intersecting generalization
of the Erd\H{o}s--Ko--Rado theorem.
Two families $\A,\B\subseteq \binom{[n]}{k}$ are said to be {\em cross-intersecting} if $A\cap B\neq \emptyset$
for any $A\in \A$ and $B\in \B$.
He proved the following tight upper bound for the product of the sizes of
two cross-intersecting families, and Matsumoto and Tokushige~\cite{Matsumoto} characterized the unique extremal family attaining the upper bound.

\begin{theorem}[Pyber~\cite{Pyber}, Matsumoto--Tokushige~\cite{Matsumoto}]\label{EKRcross}
If $n \geq 2k$ and $\A,\B\subseteq \binom{[n]}{k}$ are two cross-intersecting families, then
$$|\A||\B|\leq \binom{n-1}{k-1}^2,$$
where the equality holds for $n>2k$ if and only if $\A=\B$ is a star.
\end{theorem}


In recent years, many stability results for Erd\H{o}s--Ko--Rado type theorems have been obtained, see, e.g.,~\cite{FriedgutCom08,Keevashshadow},
and they usually assert that
a family must be close to an extremal family
when its size approximates the maximum possible size.
We refer the interested reader to the article~\cite{EllisJEMS2019} and the references therein for more background and some different   stability results of Erd\H{o}s--Ko--Rado type theorems.
However, 
it seems that there are only a few studies on the stability result for the cross-intersecting version of the Erd\H{o}s--Ko--Rado theorem.
It is noticeable that Frankl, Lee, Siggers, and
Tokushige~\cite[Theorem 1.2]{FLST}
presented a stability result for the shifted cross $t$-intersecting families.

Our first result is the following stability result of Theorem~\ref{EKRcross}, which will be employed to prove the main result of this paper.

\begin{theorem}\label{EKRcrossstabilitythm}
There exists a constant $C>1$ such that the following holds:
Suppose that $n\geq 2.07k$ and $\A,\B\subseteq \binom{[n]}{k}$ are two cross-intersecting families.
If $|\A||\B|\geq (1-\epsilon)\binom{n-1}{k-1}^2$,
where $\epsilon\leq \frac{k^2}{C^2 n ^2 }$,
then there is a star $\mathcal{S}$ such that
$|\mathcal{S} \Delta \A|\leq C \epsilon \binom{n}{k}$ and
$|\mathcal{S} \Delta \B|\leq C \epsilon \binom{n}{k}.$
\end{theorem}

Recently,
in order to get a better understanding of the problem determining Tur\'an density of the tetrahedron,
Balogh, Clemen, and Lidick\'y~\cite{Baloghsy,Balogh} introduced a new notion measuring extremality of hypergraphs, called the codegree squared sum.
Usually, one may view a family $\A \subseteq \binom{[n]}{k}$ as the $k$-uniform hypergraph (shortly $k$-graph) with the vertex set $[n]$ and the edge set $\A$.
The {\em $d$-degree vector} of a $k$-graph $\A\subseteq \binom{[n]}{k}$ is the vector
$\mathbf{v}_d(\A) \in \RR^{\binom{[n]}{d}}$
whose $U$-entry is the number of members in $\A$ containing $U$.
Observe that the  $\ell_1$-norm of the {\em codegree vector} $\mathbf{v}_{k-1}(\A)$ is equal to $k|\A|$.
The {\em codegree squared sum} of a $k$-graph $\A\subseteq \binom{[n]}{k}$ is defined to be the square of the $\ell_2$-norm of the codegree vector of $\A$, that is, $||\mathbf{v}_{k-1}(\A)||^2$,
where $||\cdot||$ denotes the $\ell_2$-norm.
Balogh, Clemen, and Lidick\'y~\cite{Baloghsy,Balogh}
determined asymptotically the codegree squared sum of tetrahedron-free $3$-graphs and some other $3$-graphs.
Thereafter, many researchers investigated extremal problems with respect to the codegree squared sum~\cite{Brooks,Wu,Cao}.
For instance, Brooks and Linz~\cite[Theorem 1.3]{Brooks} proved the following Erd\H{o}s--Ko--Rado theorem with respect to the codegree squared sum: If $n\geq 2k$ and $\A\subseteq \binom{[n]}{k}$ is an intersecting family,
then
$$||\mathbf{v}_{k-1}(\A)||^2
\leq   \binom{n-1}{k-1}\big((k-1)(n-k+ 1)+1\big),$$
where the equality holds for $n>2k$ if and only if $\A$ is a star.

For any $\A,\B\subseteq \binom{[n]}{k}$, observe that
\begin{equation}\label{innerproductform}
\big\langle \mathbf{v}_d(\A),\mathbf{v}_d(\B)\big\rangle
=\sum_{U\in \binom{[n]}{d}}
d_{\A}(U) d_{\B}(U),
\end{equation}
where $d_\mathcal{\A}(U)$ denotes the number of members in $\A$ containing $U$.
Hence,
\[
\|\mathbf{v}_{d}(\A)\|^2
=\big\langle\mathbf{v}_d(\A),\mathbf{v}_d(\A)\big\rangle.
\]
Throughout, we use the convention that $\binom{a}{b}=0$ whenever
$b<0$ or $b>a$.
We also need the following important two parameter:
\begin{align*}
\gamma(n,k,d):=&\|\mathbf{v}_{d}(\S)\|^2=
\binom{n-1}{d-1}\binom{n-d}{k-d}^{2}
+\binom{n-1}{d}\binom{n-d-1}{k-d-1}^{2},\\
\Gamma(n,k,d):= &\|\mathbf{v}_{d}(\mathcal{K}_n^k)\|^2=
\binom {n}{d}\binom{n-d}{k-d}^{2}.
\end{align*}
where $\S\subseteq \binom{[n]}{k}$ is a star and $\mathcal{K}_n^k=\binom{[n]}{k}$ is  the complete $k$-graph.
Inspired by the above results and the inner product form~\eqref{innerproductform},
we give the following main result of the present paper,
which can be viewed as an Erd\H{o}s--Ko--Rado theorem for cross-intersecting families in the Euclidean inner product.
Moreover, it contains Theorems~\ref{EKRcross} and \ref{EKRcrossstabilitythm} as special cases.

\begin{theorem}\label{EKR2normmainthm}
Let $n \geq k\geq1$ and $k\geq d \geq 0$ be integers, and let
$\A,\B\subseteq \binom{[n]}{k}$ be two cross-intersecting families.
\begin{enumerate}[label={\rm (\alph*)}]
\item {\rm(Upper bound)} We have
\[
\big\langle\mathbf{v}_{d}(\A),\mathbf{v}_{d}(\B)\big\rangle
\leq
\begin{cases}
\Gamma(n,k,d),& k\leq n<2k,\\
\gamma(n,k,d),& n\geq 2k.
\end{cases}
\]

\item {\rm(Equality cases)}
If $k\leq n<2k$, equality holds if and only if
$\A=\B=\binom{[n]}k$.
If $n>2k$, equality holds if and only if $\A=\B$ is a star.
If $n=2k$, 
then equality holds precisely as follows:
\begin{enumerate}[label={\rm (\roman*)}]
\item If $d=0$, then $|\A|= |\B|=\frac{1}{2}\binom {n}{k}$ and
    $\mathcal B=\binom{[n]}k\setminus\overline{\mathcal A}$, where $\overline{\A}:=\{[n]\setminus A:A\in\A\}$.

\item If $0<d<k$, then  $\A=\B$ is a star or the complement of a star.

\item If $d=k$, then  $\A=\B$ and $\A$ contains exactly one
member of every complementary pair, that is, $\A$ contains exactly one member of $\{C,[n]\setminus C\}$ for every $C\in \binom{[2k]}{k}$.
\end{enumerate}

\item {\rm(Stability)} There exists a constant $C>1$ such that the following holds: If $n\geq 2.07k$ and
\[
\big\langle\v_d(\A),\v_d(\B)\big\rangle
\geq (1-\epsilon)\gamma(n,k,d),
\]
where $\epsilon\leq \frac{k^2}{C^2 n ^2 }$, then there is a star
$\mathcal{S}$ such that
$|\mathcal{S} \Delta \A|\leq C \epsilon \binom{n}{k}$ and
$|\mathcal{S} \Delta \B|\leq C \epsilon \binom{n}{k}$.
\end{enumerate}
\end{theorem}

Note that $\langle\mathbf{v}_0(\A),\mathbf{v}_0(\B)\rangle=|\A||\B|$ and
$\gamma(n,k,0)=\binom{n-1}{k-1}^2$.
Thus the case $d=0$ of Theorem~\ref{EKR2normmainthm} includes Theorems~\ref{EKRcross} and~\ref{EKRcrossstabilitythm}.

By taking $\A=\B$ in Theorem~\ref{EKR2normmainthm}, we immediately obtain the following
Erd\H{o}s--Ko--Rado theorem in $\ell_2$-norm, which includes
\cite[Theorem 1.3]{Brooks} as the special case $d=k-1$.

\begin{corollary}\label{EKR2normmainthm000}
Let $n \geq k\geq1$ and $k\geq d \geq 0$ be integers, and let
$\A \subseteq \binom{[n]}{k}$ be an intersecting family.
\begin{enumerate}[label={\rm (\alph*)}]
\item {\rm(Upper bound)} We have
\[
||\mathbf{v}_{d}(\A)||^2
\leq
\begin{cases}
\Gamma(n,k,d),& k\leq n<2k,\\
\gamma(n,k,d),& n\geq 2k.
\end{cases}
\]

\item {\rm(Equality cases)}
If $k\leq n<2k$, equality holds if and only if
$\A=\binom{[n]}k$.
If $n>2k$, equality holds if and only if $\A$ is a star.
If $n=2k$, then equality holds precisely as follows:
\begin{enumerate}[label={\rm (\roman*)}]
\item If $d\in\{0,k\}$, then $\A$ contains exactly one member of every complementary pair.

\item If $0<d<k$, then $\A$ is a star or the complement of a star.
\end{enumerate}

\item {\rm(Stability)} There exists a constant $C>1$ such that the following holds: If $n\geq 2.07k$ and
\[
||\mathbf{v}_{d}(\A)||^2 \geq (1-\epsilon)\gamma(n,k,d),
\]
where $\epsilon\leq \frac{k^2}{C^2 n ^2 }$, then there is a star
$\mathcal{S}$ such that
$|\mathcal{S} \Delta \A|\leq C \epsilon \binom{n}{k}$.
\end{enumerate}
\end{corollary}


We note that, in recent independent work,  Cao, Lu, and Zhang~\cite[Theorem 1.5]{Cao} proved Corollary~\ref{EKR2normmainthm000}~(a) and~(b) in a stronger setting in $\ell_p$-norm:
Let $k\geq2$, $n\geq2k$, $1\leq d\leq k-1$, and let $p\geq2$ be real.
If
$\A\subseteq\binom{[n]}{k}$ is intersecting, then
\begin{equation}\label{eq:all-level-main}
||\mathbf{v}_{d}(\A)||_p^p
\leq \|\mathbf{v}_{d}(\S)\|_p^p
=
\binom{n-1}{d-1}\binom{n-d}{k-d}^{\!p}
+
\binom{n-1}{d}\binom{n-d-1}{k-d-1}^{\!p},
\end{equation}
where $||\cdot||_p$ denotes the $\ell_p$-norm and $\S\subseteq \binom{[n]}{k}$ is a star.
To prove it, they introduce the convex transfer principle.

The proof of Theorem~\ref{EKR2normmainthm} relies on Theorem~\ref{EKRcrossstabilitythm}
and a cross version of Bey's inequality~\cite{Bey}, that is, Theorem~\ref{innerproductABthm} and its stability result, Lemma~\ref{stabilitylemmaip}.
The proof strategy allows us to obtain exact $\ell_2$-norm bounds from
classical extremal results whenever the extremal configuration is a star.
We combine this method and the convex transfer principle by Cao, Lu, and Zhang~\cite{Cao}
to give a $t$-cluster generalization of the Erd\H{o}s--Ko--Rado theorem in the $\ell_p$-norm for $p\geq2$.



Let $2\leq t\leq k$ and $n\geq tk/(t-1)$ be integers.
A family $\{A_1,\ldots, A_t\}\subseteq \binom{[n]}{k}$ is a {\em $t$-cluster} if
$ \cap_{i=1}^t A_i=\emptyset$ and $\left|\cup_{i=1}^t A_i\right|\leq 2k$.
A family $\mathcal{A} \subseteq \binom{[n]}{k}$ is called {\em $t$-cluster-free} if it contains no $t$-clusters.
Note that a family $\mathcal{A} \subseteq \binom{[n]}{k}$ is intersecting if and only if it is $2$-cluster-free.
Thus, we can restate the Erd\H{o}s--Ko--Rado theorem as saying that any $2$-cluster-free family has size at most $\binom{n-1}{k-1}$.
Mubayi~\cite{Mubayi2006} conjectured that any $t$-cluster-free family must have size at most $\binom{n-1}{k-1}$.
He solved this conjecture for $t=3$ in~\cite{Mubayi2006} and later proved  a stability result for every $t\geq 2$ ~\cite[Theorem 2]{Mubayi2007}.
The conjecture was completely solved by Currier~\cite[Theorem 2]{Currier}.
We combine these two result into the following theorem and refer the reader to~\cite{Currier,Mubayi2006,Mubayi2007} for more background and related results.

\begin{theorem}[\cite{Currier,Mubayi2007}]\label{dclusterEKR}
Let $k \geq   t\geq 2$ and $n\geq tk/(t-1)$ be integers.
Let $\A \subseteq \binom{[n]}{k}$ be a  $t$-cluster-free family.
\begin{enumerate}[label={\rm (\alph*)}]

\item {\rm(Upper bound)}
We have
\[
|\A|\leq\binom{n-1}{k-1}.
\]

\item {\rm(Uniqueness)}
If $(t,n)\neq(2,2k)$, then equality holds if and only if
$\A$ is a star.

\item {\rm(Stability)}
For every $\delta>0$, there exist
$\epsilon=\epsilon(\delta,t,k)>0$ and $n_0=n_0(\epsilon,t,k)$
such that the following holds for every $n>n_0$:
If
\[
|\A|\geq
(1-\epsilon)\binom{n-1}{k-1},
\]
then there exists an
$S\in\binom{[n]}{n-1}$ such that
\[
\left|\A\cap\binom{S}{k}\right|
<
\delta\binom{n-1}{k-1}.
\]

\end{enumerate}
\end{theorem}

%
%
%
%

Brooks and Linz~\cite[Theorem 3.6]{Brooks} proved that if $n\geq tk/(t-1)$, $k>t>1$, and $\A\subseteq\binom{[n]}k$ is $t$-cluster-free, then
$$||\mathbf{v}_{k-1}(\A)||^2
\leq   \binom{n-1}{k-1}\big((k-1)(n-k+ 1)+1\big). $$
Our third result provides an $\ell_p$-norm extension of Theorem~\ref{dclusterEKR}, which includes \cite[Theorem 3.6]{Brooks} and~\cite[Theorem 1.5]{Cao} as special cases.


\begin{theorem}
\label{clusterEKRlp}
Let $k \geq t\geq 2$, $k\geq d\geq 0$, and $n\geq tk/(t-1)$ be integers.
Let $p\geq2$ be real and let
$\A\subseteq\binom{[n]}k$ be a $t$-cluster-free family.
\begin{enumerate}[label={\rm (\alph*)}]

\item {\rm(Upper bound)} We have
\begin{equation}\label{eq:cluster-lp-bound}
\|\mathbf{v}_d(\A)\|_p^p
\leq \binom{n-1}{d-1}\binom{n-d}{k-d}^{\!p}
+\binom{n-1}{d}\binom{n-d-1}{k-d-1}^{\!p}=: \Phi_p(n,k,d).
\end{equation}

\item {\rm(Equality cases)}
\begin{enumerate}[label={\rm (\roman*)}]
\item If $d\in\{0,k\}$, then equality in~\eqref{eq:cluster-lp-bound}
holds if and only if
\[
|\A|=\binom{n-1}{k-1}.
\]
Consequently, if $(t,n)\neq(2,2k)$, equality holds if and only if
$\A$ is a star.  If $(t,n)=(2,2k)$, equality holds if and only if
$\A$ contains exactly one member of every complementary pair
$\{F,[n]\setminus F\}$.

\item Suppose that $1\leq d\leq k-1$ and $p=2$.  If
$(t,n)\neq(2,2k)$, equality holds if and only if $\A$ is a star.
If $(t,n)=(2,2k)$, equality holds if and only if $\A$ is a star
or the complement $\binom{[n]}k\setminus\mathcal S_x$ of a star
$\mathcal S_x$.

\item Suppose that $1\leq d\leq k-1$ and $p>2$.  Equality holds
if and only if $\A$ is a star.
\end{enumerate}

\item {\rm(Stability)} Fix $t,k,d,p$ as above.  For every
$\delta>0$, there exist
$\epsilon=\epsilon(\delta,t,k,d,p)>0$ and
$n_0=n_0(\delta,t,k,d,p)$ such that the following holds whenever
$n>n_0$.  If $\A\subseteq\binom{[n]}k$ is $t$-cluster-free and
\[
\|\v_d(\A)\|_p^p\geq
(1-\epsilon)\Phi_p(n,k,d),
\]
then there exists an $S\in\binom{[n]}{n-1}$ such that
\[
\left|\A\cap\binom Sk\right|
<\delta\binom{n-1}{k-1}.
\]
\end{enumerate}
\end{theorem}




The rest of this paper is organized as follows.
The next section provides a stability result for the cross-independent version of the Hoffman--Delsarte bound.
Based on this result, in Section 3, we present a proof of Theorem~\ref{EKRcrossstabilitythm}.
Section 4 shows the proofs of Theorems~\ref{EKR2normmainthm} and~\ref{clusterEKRlp}.
We conclude this paper in Section~\ref{ConclusionSec}.


\section{A stability result for the Hoffman--Delsarte bound}

To prove Theorem \ref{EKRcrossstabilitythm}, we introduce
a stability result for the cross-independent version of the Hoffman--Delsarte bound.

Let $G$ be an $N$-vertex graph.
A real symmetric $N\times N$ matrix $M=(M_{ij})$ is called a {\em pseudo adjacency matrix} of $G$ if $M_{ij}= 0$ whenever $\{i, j\}\notin E(G)$,
and the all-one vector $\mathbf{1}$ is an eigenvector of $M$ with a positive eigenvalue.
Let $M$ be a pseudo adjacency matrix of $G$
with eigenvalues $\lambda_1, \lambda_2, \ldots , \lambda_N$,
where $\lambda_1$ is the positive eigenvalue of $M$
corresponding to the eigenvector $\mathbf{1}$.
We write $\lambda(M)=\lambda_1$ and  $\mu (M)=\max_{2\leq i\leq N} |\lambda_i|$.


Ellis, Friedgut, and Pilpel~\cite{EllisJAMS}
proved an Erd\H{o}s--Ko--Rado theorem for cross $t$-intersecting permutations.
To achieve it, they used the following cross-independent version of the Hoffman--Delsarte bound. 
This bound is useful when we study the cross-independent type problems,
see, e.g.,~\cite{Tokushige} for many applications.

\begin{theorem}[Ellis--Friedgut--Pilpel \cite{EllisJAMS}]\label{CrossHDbound}
Let $G$ be an $N$-vertex graph with a pseudo adjacency matrix $M$.
Suppose that $A$ and $B$ are sets of vertices in $G$ such
that there are no edges of $G$ between $A$ and $B$.
Then
$$ |A||B| \leq \left(\frac{\mu (M)}{\lambda (M)+\mu (M)}\right)^2N^2.$$
\end{theorem}

The following inequality from the AM--GM inequality will be used in later.

\begin{lemma}\label{specialAGinq}
For all $x_1\geq y_1\geq 0$ and $x_2\geq y_2\geq 0$, we have
$$\sqrt{(x_1-y_1)(x_2-y_2)}\leq \sqrt{x_1x_2}-\sqrt{y_1y_2},$$
where the equality holds if and only if $x_1y_2=x_2y_1$.
\end{lemma}
\begin{proof}
Using the AM--GM inequality, we have
\begin{align*}
(x_1-y_1)(x_2-y_2)=&x_1x_2 +y_1y_2- (x_1y_2+x_2 y_1)\\
\leq &x_1x_2 +y_1y_2-2\sqrt{x_1x_2y_1y_2}
= (\sqrt{x_1x_2}-\sqrt{y_1y_2})^2
\end{align*}
where the equality holds if and only if $x_1y_2=x_2y_1$, as desired.
\end{proof}

Usually, we write $\1_A\in \RR^S$ for the {\em characteristic vector} of a subset $A$  of the set $S$, that is, the vector whose $u$-entry is $1$ if $u\in A$ and $0$ otherwise.
We next give a stability result for Theorem~\ref{CrossHDbound}, where the basic idea of the proof comes from \cite[Lemma 2.1]{Berger} by Berger and Zhao.

\begin{lemma}\label{CrossHDboundStability}
Let $M$ be a pseudo adjacency matrix of an $N$-vertex graph $G$
with eigenvalues $\lambda_1=\lambda(M), \lambda_2, \ldots , \lambda_N$ and a corresponding real orthonormal eigenbasis
$\v_1,\v_2, \ldots ,\v_N$, where $\v_1=N^{-1/2}\1$.
Assume that $\mu(M)>0$.
Suppose that $A$ and $B$ are sets of vertices in $G$ such
that there are no edges of $G$ between $A$ and $B$.
For any $\epsilon \in [0,1)$ and $\xi\in (0,1]$,
if
$$ |A||B| \geq (1-\epsilon)\left(\frac{\mu (M)}{\lambda (M)+\mu (M)}\right)^2N^2,$$
then the projection $P_{\1_A}$ of $\1_A$ onto the eigenspace spanned by eigenvectors $\{\v_i: 2\leq i\leq N ~with~ |\lambda_i|\leq (1-\xi)\mu(M) \}$
satisfies
satisfies
$$||P_{\1_A}||^2\leq
\frac{\big(\lambda(M) + \mu(M)\big)^2 }{ \xi(2-\xi)\lambda(M) ^2} \epsilon  N.$$
\end{lemma}
\begin{proof}
Set $\mu=\mu(M)$, $\lambda=\lambda_1$, $\1_A=\sum_{i=1}^N a_i\v_i$, and $\1_B=\sum_{i=1}^N b_i\v_i$.
The assumptions $\mu>0$ and $\epsilon<1$ imply that $|A||B|>0$.
Observe that
$$ ||P_{\1_A}||^2=\sum_{i>1: |\lambda_i|\leq (1-\xi)\mu } a_i^2=:D.$$
For $2\leq i\leq N$, write
$$\widehat{a}_i=
\left\{
\begin{array}{ll}
a_i, &  \mbox{if~} |\lambda_i|>(1-\xi) \mu ;\\
(1-\xi)a_i, & \mbox{otherwise}.
\end{array}\right.$$
Then one may check that
$$\sum_{ i>1} |\widehat{a}_i|^2=
\left(\sum_{ i>1: |\lambda_i|>(1-\xi) \mu  }  |a_i|^2\right)+
\left(\sum_{ i>1: |\lambda_i|\leq(1-\xi) \mu  }  |a_i|^2 (1-\xi)^2\right)
=\left(\sum_{ i>1}|a_i|^2\right) - D \xi(2-\xi).$$
Note that $\sum_{ i>1}|a_i|^2 \leq \sum_{ i=1}^N|a_i|^2=|A|\leq N$,
so we have
\begin{equation}\label{Hfmanstabilityeq111}
\sum_{ i>1} |\widehat{a}_i|^2\leq \left(\sum_{ i>1}|a_i|^2\right) - \frac{\sum_{ i>1}|a_i|^2}{N} D \xi(2-\xi)=
\left(1-\frac{D \xi(2-\xi)}{N}\right)\sum_{ i>1}|a_i|^2.
\end{equation}

Since there are no edges between $A$ and $B$, we get
$$0=\1_A^{\mathsf T}M\1_B=\left(\sum_{i=1}^N a_i\v_i\right)^{\!\mathsf T}
\left(M\sum_{i=1}^N b_i \v_i\right)
=\sum_{i=1}^N a_ib_i \lambda_i,$$
which implies that
\begin{align*}
a_1b_1\lambda & = -\sum_{i=2}^N a_ib_i \lambda_i
\leq \left(\sum_{ i>1: |\lambda_i|>(1-\xi) \mu  }  |a_i||b_i|\mu\right) +
\left(\sum_{ i>1: |\lambda_i|\leq(1-\xi) \mu  }  |a_i||b_i| (1-\xi) \mu\right) \\
&= \sum_{ i>1} |\widehat{a}_i||b_i|   \mu
\leq  \mu  \left(\sum_{ i>1} |\widehat{a}_i|^2\right)^{\frac{1}{2}}\left(\sum_{ i>1} |b_i|^2\right)^{\frac{1}{2}},
\end{align*}
where the last inequality follows from the Cauchy--Schwarz inequality.
Set $\gamma^2= 1-\frac{D \xi(2-\xi)}{N} $.
Since $0\leq D\leq |A|\leq N$ and $0<\xi(2-\xi)\leq1$,
we have $0\leq\gamma\leq1$.
Applying \eqref{Hfmanstabilityeq111} to the preceding inequality, we deduce that
\begin{equation}\label{Hfmanstabilityeq222}
a_1b_1\lambda
\leq   \gamma \mu  
\left(\sum_{ i>1} |a_i|^2\right)^{\frac{1}{2}}\left(\sum_{ i>1} |b_i|^2\right)^{\frac{1}{2}} .
\end{equation}
Observe that
$$a_1=\big\langle \1_A,\1\big/\sqrt{N}\big\rangle=|A| \big/\sqrt{N},
~ \left(\sum_{ i=1}^N |a_i|^2\right)-|a_1|^2=|A|-|A|^2\big/ N,$$
$$b_1=\big\langle \1_B,\1\big/\sqrt{N}\big\rangle=|B| \big/\sqrt{N},
~{\rm and~} \left(\sum_{ i=1}^N |b_i|^2\right)-|b_1|^2=|B|-|B|^2\big/ N.$$
Substituting these terms into \eqref{Hfmanstabilityeq222} and applying Lemma \ref{specialAGinq}, we get that
\begin{equation}\label{Hfmanstabilityeq333}
\frac{|A||B|}{N} \lambda
\leq \gamma  \mu
\left(|A|-|A|^2 / N \right)^{\frac{1}{2}}
\left(|B|-|B|^2 / N \right)^{\frac{1}{2}}
\leq  \gamma \mu
\left( \big(|A||B|\big)^{\frac{1}{2}}- \frac{|A||B|}{N} \right),
\end{equation}
which implies that
$$
\frac{|A||B|}{N^2}
\leq \left(\frac{\gamma\mu }{\lambda  +\gamma \mu }\right)^2.
$$
Combining this inequality and the assumption,
one may calculate that
\begin{align*}
\epsilon \left(\frac{\mu }{\lambda  +\mu }\right)^2
\geq &\left(\frac{ \mu  }{\lambda  +  \mu }\right)^2
-\left(\frac{\gamma\mu }{\lambda  +\gamma \mu  }\right)^2
=\frac{ \mu^2 \big(\lambda  +  \gamma\mu \big)^2
- \gamma^2\mu^2 (\lambda  +  \mu)^2}
{\big(\lambda +  \mu  \big)^2\big(\lambda  +  \gamma\mu  \big)^2}
\\
=&\frac{\lambda^2\mu^2\big(1-\gamma ^2\big)+2\lambda\mu^3(\gamma-\gamma ^2)}
{\big(\lambda  +  \mu  \big)^2\big(\lambda  +  \gamma\mu  \big)^2}
\geq \frac{ \lambda^2\mu^2\big(1- \gamma ^2\big) }
{\big(\lambda  +  \mu  \big)^4}=\frac{ \lambda^2\mu^2 }
{\big(\lambda  +  \mu  \big)^4}\frac{D \xi(2-\xi)}{N}.
\end{align*}
Simplifying this inequality, we obtain
$$\frac{D }{N} \leq
\frac{\big(\lambda +  \mu  \big)^2 }{ \xi(2-\xi)\lambda ^2} \epsilon,$$
as desired.
\end{proof}

\section{Proof of Theorem \ref{EKRcrossstabilitythm}}

The aim of this section is to prove Theorem \ref{EKRcrossstabilitythm}.
The proof is based on an analytic result on Boolean functions by Filmus \cite{Filmus} and Lemma \ref{CrossHDboundStability}.

For a  family  $\A \subseteq \binom{[n]}{k}$, its {\em characteristic function}
$f_\A:\binom{[n]}{k}\to \{0,1\}$ is a Boolean function such that
$f_\A(A)$ is $1$ if $A\in\A$ and $0$ otherwise.
It will be convenient to identify $\A$, $\1_\A$, and $f_\A$ if there is no ambiguity.
A function $f :\binom{[n]}{k}\to \{0,1\}$ is said to be {\em affine}
if it can be written as 
$$c_0 +\sum_{i=1}^n c_i f_{\S_i}$$
for some constants $c_i$, $1\leq i \leq n$, where
$\S_i=\{S\in \binom{[n]}{k}: i\in S\}$ is the star.
Two functions $f,g:\binom{[n]}{k} \to \{0, 1\}$ are said to be
{\em $\epsilon$-close} if
$$||f-g||^2=\sum_{S \in \binom{[n]}{k}} \big(f(S)-g(S)\big)^2
\leq \epsilon \binom{n}{k}.$$
Clearly, if the characteristic functions of two families
$\A, \B\subseteq \binom{[n]}{k}$ are $\epsilon$-close, then their symmetric difference satisfies $|\A \Delta \B|\leq \epsilon \binom{n}{k}.$

Filmus \cite[Theorem 3.1]{Filmus} proved an analogous theorem of the Friedgut--Kalai--Naor theorem \cite{Friedgut} on the slice $\binom{[n]}{k}$.
Our proof of Theorem \ref{EKRcrossstabilitythm} relies on the following special case (cf. \cite[Theorem~6]{Das}) of Filmus's result.
Using it,
Das and Tran~\cite[Theorem 2]{Das} presented a removal lemma showing that large families with few disjoint pairs must be close to a union of stars.
For more related results and background, we refer to~\cite{Filmus,FilmusCPC}.

\begin{theorem}[Filmus \cite{Filmus}]\label{FNKthninkslice}
There exists a constant $C>1$ such that the following holds:
Suppose that $2\leq k\leq n/2$ and $\varepsilon<  \frac{k}{128n}$.
If $f:\binom{[n]}{k}\to \{0,1\}$ is $\varepsilon$-close to an affine function,
then there exists $S\subseteq [n]$ of size
$|S|\leq \max\left\{1, C \sqrt{\varepsilon}n/k \right\}$ such that either $f$ or $1-f$ is $(C\varepsilon)$-close to $\max_{i\in S} x_i$,
where $\max_{i\in \emptyset} x_i=0$ and $\max_{i\in S} x_i$ denotes the characteristic function of
the union of stars $\S_i$ with $i\in S$.
\end{theorem}

The following lemma will be used in the proof of Theorem
\ref{EKRcrossstabilitythm}.

\begin{lemma}\label{crossstarlemma}
Let $n \geq 2k$ and $k\geq 2$, and let $\mathcal{S}_i=\{S\in \binom{[n]}{k}: i\in S\}$.
If $i\neq j$ and $\A,\B\subseteq \binom{[n]}{k}$ are two cross-intersecting families, then
$$|\A\cap \mathcal{S}_i|  |\B\cap \mathcal{S}_j|
\leq \binom{n-2}{k-2}\binom{n-1}{k-1}+\binom{n-3}{k-2}^2.$$
\end{lemma}
\begin{proof}
Write $\mathcal{S}_{i,j}=\{S\in \binom{[n]}{k}: i,j\in S\}$ and $\mathcal{S}_{i,\overline{j}}=
\{S\in \binom{[n]}{k}: i \in S {~\rm and~} j\notin S\}$.
Observe that $\A(i,\overline{j}):=\{A\setminus \{i \} : A\in \A\cap \mathcal{S}_{i,\overline{j}}\}$ and
$\B(j,\overline{i}):=\{B\setminus \{ j\} : B\in \B\cap \mathcal{S}_{j,\overline{i}}\}$ are cross-intersecting families on $\binom{[n]\setminus \{i,j\}}{k-1}$.
Applying Theorem \ref{EKRcross}, we have
$$|\A\cap \mathcal{S}_{i,\overline{j}} |
|\B\cap \mathcal{S}_{j,\overline{i}} |
=|\A(i,\overline{j})|
|\B(j,\overline{i})|
\leq \binom{n-3}{k-2}^2.$$
Recall a theorem of Hilton \cite[Theorem 1]{Hilton} which states that
$|\F|+|\G| \leq \binom{n}{k}$ when $\F,\G\subseteq \binom{[n]}{k}$ are
two cross-intersecting families.
Using it, we get that
$$|\A\cap \mathcal{S}_{i,\overline{j}} | +
|\B\cap \mathcal{S}_{j,\overline{i}} |
=|\A(i,\overline{j})|+|\B(j,\overline{i})|
\leq \binom{n-2}{k-1}.$$
Combining these, we deduce that
\begin{align*}
|\A\cap \mathcal{S}_i|  |\B\cap \mathcal{S}_j|
&=\left(|\A\cap \mathcal{S}_{i,j}|+|\A\cap \mathcal{S}_{i,\overline{j}}|\right) \left(|\B\cap \mathcal{S}_{i,j}|+|\B\cap \mathcal{S}_{j,\overline{i}}|\right) \\
&\leq \binom{n-2}{k-2}^2+ \binom{n-2}{k-2}
\left(|\A\cap \mathcal{S}_{i,\overline{j}} | +
|\B\cap \mathcal{S}_{j,\overline{i}} |\right)
+|\A\cap \mathcal{S}_{i,\overline{j}} |
|\B\cap \mathcal{S}_{j,\overline{i}} |\\
&\leq \binom{n-2}{k-2}^2+\binom{n-2}{k-2}\binom{n-2}{k-1}+ \binom{n-3}{k-2}^2\\
&=\binom{n-2}{k-2}\binom{n-1}{k-1}+\binom{n-3}{k-2}^2,
\end{align*}
as desired.
\end{proof}

Combining Theorem \ref{FNKthninkslice} and Lemma \ref{CrossHDboundStability},
we are now ready to prove Theorem \ref{EKRcrossstabilitythm} which gives
a stability result of Theorem \ref{EKRcross}.

\begin{proof}[{\bf Proof of Theorem \ref{EKRcrossstabilitythm}}]
The statement is trivial for $k=1$.
Given integers $2\leq  k$ and $n\geq 2.07k$,
the {\em Kneser graph} $K(n, k)$ is the graph
on the vertex set $\binom{[n]}{k}$,
where two vertices 
form an edge if and only if their intersection is empty.
Let $A$ be the adjacency matrix of $K(n,k)$.
Lov\'asz \cite{Lovasz} showed the eigenvalues of $A$ are
$$\lambda_i=(-1)^i \binom{n-k-i}{k-i},$$
with multiplicities $\binom{n}{i}-\binom{n}{i-1}$,
where $0\leq i \leq k$ and $\binom{n}{ -1}=0$.
Moreover, the eigenspace $E_0$ of $\lambda_0$ is spanned by $\1$,
and the eigenspace $E_1$ of $\lambda_1$ is spanned by the vectors
$\1_{\S_i}-\frac{k}{n}\1$ with $i\in [n]$, which has dimension $n-1$.

Let $C_0>1$ be the constant in Theorem \ref{FNKthninkslice}
and set $C=64C_0$ which is the required constant in the statement.
Suppose that $\A,\B\subseteq \binom{[n]}{k}$ are two cross-intersecting families such that $|\A||\B|\geq (1-\epsilon)\binom{n-1}{k-1}^2$,
where $\epsilon\leq \frac{k^2}{C^2n^2}$.


\begin{claim}\label{claimprojectionsmall}
The projection $P_{\1_\A}$ of $\1_\A$ onto $(E_0\oplus E_1)^{\bot}$
satisfies
$$||P_{\1_\A}||^2 < \frac{n^2}{ (n-2 )(n-2k )}\epsilon N.$$
\end{claim}
\begin{proof}
Since $\A,\B$ are cross-intersecting,
there are no edges of $K(n, k)$ between them.
Observe that $\lambda(A)=\lambda_0=\binom{n-k }{k}$ and
$\mu(A)=|\lambda_1|= \binom{n-k-1}{k-1}$,
so we have
$$ \left(\frac{\mu (A)}{\lambda (A)+\mu (A)}\right)^2N^2=\binom{n-1}{k-1}^2.$$
Combining these and the hypothesis,
we may apply Lemma \ref{CrossHDboundStability} to $A$ as follows.
Note that
$|\lambda_2|=\max_{2\leq i\leq k} |\lambda_i|=\binom{n-k-2}{k-2}$.
We choose $\xi$ such that $(1-\xi)|\lambda_1|= |\lambda_2|,$ that is, $\xi=\frac{ n-2k }{n-k-1}$.
Then Lemma \ref{CrossHDboundStability} implies that
$$||P_{\1_\A}||^2 \leq \frac{\big(1 +  \mu(A)/ \lambda(A) \big)^2 }
{ \xi(2-\xi) } \epsilon N
=\left(\frac{n}{n-k}\right)^2 \frac{(n-k-1)^2}{(n-2)(n-2k)} \epsilon N
< \frac{n^2}{ (n-2 )(n-2k )}\epsilon N,$$
as desired.
\end{proof}

\begin{claim}\label{claimcloseaffine}
The characteristic function $f_\A$ of $\A$ is $64\epsilon$-close to an affine function.
\end{claim}
\begin{proof}
Let $\f_0$ and $\f_1$ be the projections of $\1_\A$ onto subspaces $E_0$ and
$E_1$, respectively.
Then the function $\f_0+\f_1$ is affine by the structures of $E_0$ and
$E_1$.
Moreover, by the orthogonality of the eigenspaces and Claim \ref{claimprojectionsmall},
we deduce that
\begin{align*}
||\1_\A-(\f_0+\f_1)||^2=&||P_{\1_\A}||^2
< \frac{n^2}{ (n-2 )(n-2k )}\epsilon N
=\frac{1}{ \left(1-\frac{2}{n} \right)\left(1-\frac{2k}{n}\right)}\epsilon N\\
\leq&\frac{1}{\left(1-\frac{2}{4.14} \right) \left(1-\frac{2k}{2.07k}\right)}\epsilon N
<64\epsilon N ,
\end{align*}
%
where the last inequality follows from the fact that $n\geq 2.07k\geq 4.14.$
Thus, $\1_\A$ is $64\epsilon$-close to the affine function $\f_0+\f_1$,
as desired.
\end{proof}

Write $\varepsilon=64\epsilon$. 
Recall that $\epsilon\leq \frac{k^2}{C^2n^2}$ and $C=64C_0$, so
$$\varepsilon=64\epsilon \leq \frac{64k^2}{(64C_0 n)^2}  < \frac{k}{128n}$$
and
$$\frac{C_0 \sqrt{\varepsilon}n}{k}=
\frac{C_0n\sqrt{64\epsilon}}{k}
\leq\frac{C_0n }{k}\frac{k }{8C_0n}=\frac18<1.$$
By Claim \ref{claimcloseaffine}, $f_\A$ is $64\epsilon$-close to an affine function.
Hence, applying Theorem \ref{FNKthninkslice} with $\varepsilon=64\epsilon$ and
$C_0\varepsilon=C_064\epsilon=C\epsilon$,
we conclude that there exist $i$
such that $f_\A$ is $(C\epsilon)$-close to one of $\{0,1, f_{\S_i},1-f_{\S_i}\}$.
By the symmetry of $\A$ and $\B$, we also get that there exist $j$ such that $f_\B$ is $(C\epsilon)$-close to one of $\{0,1, f_{\S_j},1-f_{\S_j}\}$.

Note that $C\epsilon\leq \frac{Ck^2}{C^2n^2} <\frac{k^2}{64n^2}$.
If $f_\A$ is $(C\epsilon)$-close to $0$, then 
$$|\A| |\B|= ||\1_\A||^2  ||\1_\B||^2 \leq C\epsilon N \cdot N
<\frac{1}{64} \binom{n-1}{k-1}^2 <(1-\epsilon)\binom{n-1}{k-1}^2,$$
a contradiction to the hypothesis.
So $f_\A$ must be $(C\epsilon)$-close to one of
$\{f_{\S_i},1-f_{\S_i},1\}$.
If $f_\B$ is $(C\epsilon)$-close to one of $\{1-f_{\S_j},1\}$, then
\begin{align}\nonumber
|\A| |\B|
\geq &
\left( \binom{n-1}{k-1} -C\epsilon N  \right)\left( \binom{n-1}{k } -C\epsilon N  \right)\\\nonumber
\geq&\binom{n-1}{k-1}\binom{n-1}{k }
- C\epsilon N\binom{n-1}{k-1} - C\epsilon N \binom{n-1}{k } \\\label{caseSandC}
=&  \binom{n-1}{k-1}^2 \left(\frac{n-k}{k}-C\epsilon\frac{n^2}{k^2}   \right).
\end{align}
As $n\geq 2.07k$, we get that 
$$\frac{n-k}{k}- C\epsilon\frac{n^2}{k^2}
\geq \frac{1.07k}{k}- C\frac{k^2}{  C^2 n ^2 }\frac{n^2}{k^2}
=  1.07- \frac{1}{C}\geq 1.07- \frac{1}{64}>1.$$
Applying this estimate to \eqref{caseSandC}, we obtain $|\A| |\B|>\binom{n-1}{k-1}^2$,
a contradiction to Theorem \ref{EKRcross}.

By the above discussion and the symmetry of $\A$ and $\B$,
we obtain that $f_\A$ is $(C\epsilon)$-close to $f_{\S_i}$
and $f_\B$ is $(C\epsilon)$-close to $f_{\S_j}$.
If $i\neq j$, by Lemma \ref{crossstarlemma}, then we get
\begin{align*}
|\A| |\B|\leq & \big(| \A\cap \mathcal{S}_i|+ C\epsilon N\big)
\big(|\B \cap \mathcal{S}_j| + C\epsilon N\big)\\
\leq & | \A\cap \mathcal{S}_i| | \B\cap \mathcal{S}_j|
+ 2C\epsilon N \binom{n-1}{k-1}+ (C\epsilon N)^2\\
\leq& \binom{n-2}{k-2}\binom{n-1}{k-1}+\binom{n-3}{k-2}^2
+ \frac{2}{64}\binom{n-1}{k-1}^2 + \left(\frac{1}{64 }\binom{n-1}{k-1} \right)^2\\
<& \binom{n-1}{k-1}^2\left(\frac{k-1}{n-1}
+ \left(\frac{(k-1)(n-k)}{(n-1)(n-2)}\right)^2+ \frac{3}{64} \right)\\
\leq & \binom{n-1}{k-1}^2\left(\frac{k-1}{n-1}
+ \left(\frac{k-1}{n-1}\right)^2+ \frac{3}{64} \right)\\
< & \binom{n-1}{k-1}^2\left(\frac{1}{2}+\frac{1}{4} + \frac{3}{64} \right) <(1-\epsilon)\binom{n-1}{k-1}^2,
\end{align*}
a contradiction to the hypothesis.
So we conclude that $i=j$, and hence that $f_\A$ and $f_\B$ are both $(C\epsilon)$-close to $f_{\S_i}$.
The proof is completed.
\end{proof}

\begin{remark}
We note that the constant $2.07$ in the theorem is not optimal for Theorem \ref{EKRcrossstabilitythm}.
According to the proof of Theorem \ref{EKRcrossstabilitythm},
if we choose a larger universal constant $C$, then we can obtain a smaller constant closer to $2$ than $2.07$.
\end{remark}

\section{Proofs of Theorems~\ref{EKR2normmainthm} and~\ref{clusterEKRlp}}
\label{sec:proof-main}

This section is devoted to proving~Theorems~\ref{EKR2normmainthm} and~\ref{clusterEKRlp}.
The proof relies on the eigenvalues of the matrices
of the Johnson scheme and Theorem~\ref{EKRcrossstabilitythm}.


Fix integers $n>k$ and set
\[
N=\binom{n}{k}
\qquad\text{and}\qquad
m=\min\{k,n-k\}.
\]
For $0\leq j\leq m$, let $A_j$ be the matrix indexed by
$\binom{[n]}{k}$ whose $(S,T)$-entry is $1$ if $|S\cap T|=k-j$, and is
$0$ otherwise.  The matrices $A_0,A_1,\ldots,A_m$ form the Johnson
association scheme $J(n,k)$.
It is known that, see, e.g.,~\cite{Delsarte} or~\cite[Chapter 6]{GodsilEKRbook}, there is an orthogonal decomposition
\[
\mathbb{R}^{N}=E_0\oplus E_1\oplus\cdots\oplus E_m
\]
into common eigenspaces, where
\[
\dim E_i=\binom{n}{i}-\binom{n}{i-1}
\quad (0\leq i\leq m),
\]
with the convention $\binom{n}{-1}=0$.  In particular,
$E_0=\langle\mathbf{1}\rangle$, and
\[
E_0\oplus E_1
=\operatorname{span}\{\mathbf{1}_{\mathcal{S}_r}:r\in[n]\},
\qquad
\mathcal{S}_r=\left\{S\in\binom{[n]}{k}:r\in S\right\}.
\]
The eigenvalue of $A_j$ on $E_i$ is
\[
\theta_j(i)
=\sum_{s=0}^{j}(-1)^{j-s}
\binom{k-s}{j-s}\binom{k-i}{s}
\binom{n-k+s-i}{s}.
\]
Notice that when $n<2k$ the decomposition stops at $E_{n-k}$, not at
$E_k$.

For $0\leq d\leq k$, define
\[
M_d=\sum_{j=0}^{m}\binom{k-j}{d}A_j.
\]
Bey~\cite{Bey} proved that the eigenvalues of $M_d$ are given by
\begin{equation}\label{eigenofMd}
\lambda_d(i)
:=\sum_{j=0}^{m}\binom{k-j}{d}\theta_j(i)
=\binom{k-i}{k-d}\binom{n-d-i}{k-d},
\qquad 0\leq i\leq m.
\end{equation}
If $0<d<k$, then
\[
\lambda_d(0)>\lambda_d(1)>\lambda_d(i)\geq0
\qquad (2\leq i\leq m).
\]
Indeed,
$\lambda_d(i)=0$ for $i>d$, and for
$0\leq i<\min\{d,m\}$,
\[
\frac{\lambda_d(i+1)}{\lambda_d(i)}
=\frac{d-i}{k-i}\frac{n-k-i}{n-d-i}<1,
\]
which proves the assertion.

The following calculation on $\lambda _d(0)$ and $\lambda _d(1)$ will be used in the next proofs.

\begin{lemma}\label{eigenvalue01ofMd}
Let $n>k>d>0$.  Then
\[
\frac{\lambda_d(0)-\lambda_d(1)}{\binom{n}{k}}
=\frac{\binom{k}{d}\binom{k-1}{d}}{\binom{n-1}{d}}.
\]
\end{lemma}

\begin{proof}
By~\eqref{eigenofMd}, one may check that
\begin{align*}
\lambda _d(0)-\lambda _d(1) =&\binom{k }{ d}
\binom{n-d}{k-d}
-   \binom{k-1}{ d-1}
\binom{n-d-1}{k-d}
= \binom{k}{d}\binom{n-d}{k-d}
\left(1- \frac{d}{k} \frac{n-k}{n-d} \right)\\
=&  \binom{k}{d}\binom{n-d}{k-d}
\frac{n(k-d)} {k(n-d)}
\overset{(\star)}{=}\binom{k}{d} \binom{n}{k} \frac{\binom{k}{d} \frac{k-d}{k}} {\binom{n}{d}\frac{n-d}{n}   }
=\binom{k}{d}\binom{n}{k} \frac{\binom{k-1}{d}}{\binom{n-1}{d}},
\end{align*}
as desired,
where the equality ($\star$) follows from the identity
\begin{align*}\binom{n}{k}\binom{k}{d}
=\binom{n}{d}\binom{n-d}{k-d}.  \tag*{\qedhere}
\end{align*}
\end{proof}

Generalizing a result of de Caen~\cite[Theorem 1]{Caen} from graphs to hypergraphs,
Bey~\cite[Theorem 1]{Bey} presented a tight upper bound for
$||\mathbf{v}_{d}(\A)||^2$ of a $k$-graph $\A$ and characterized all extremal families.
This result is the main tool in the work of~\cite{Brooks}.
To prove Theorem \ref{EKR2normmainthm}, we need the following result, which
coincides with Bey's result when $\A=\B$.

\begin{theorem}\label{innerproductABthm}
Let $n\geq k\geq d\geq0$, and let
$\mathcal{A},\mathcal{B}\subseteq\binom{[n]}{k}$.
If $n=k$, then
\begin{equation}\label{eq:innerproduct-n-equals-k}
\big\langle\mathbf{v}_d(\mathcal{A}),
\mathbf{v}_d(\mathcal{B})\big\rangle
=\binom kd|\mathcal{A}|\,|\mathcal{B}|.
\end{equation}
If $n>k$, then
\begin{equation}\label{innerproductABthminq}
\big\langle\mathbf{v}_d(\mathcal{A}),
\mathbf{v}_d(\mathcal{B})\big\rangle
\leq
\frac{\binom{k}{d}\binom{k-1}{d}}{\binom{n-1}{d}}
|\mathcal{A}|\,|\mathcal{B}| +\binom{k-1}{d-1}\binom{n-d-1}{k-d}
\sqrt{|\mathcal{A}|\,|\mathcal{B}|}.
\end{equation}
For $n>k$, equality is automatic when $d=0$,
and equality
holds if and only if at least one family is empty or
$\mathcal{A}=\mathcal{B}$ when $d=k$.
Suppose in addition that $0<d<k$ and
$|\mathcal{A}|\,|\mathcal{B}|>0$.
Equality holds in
\eqref{innerproductABthminq} if and only if
$\mathcal{A}=\mathcal{B}=:\mathcal{F}$ and one of the following holds:
\begin{enumerate}[label={\rm (\roman*)}]
\item $n=k+1$, in which case $\mathcal{F}$ may be any nonempty
subfamily of $\binom{[n]}{k}$;
\item $n\geq k+2$, in which case $\mathcal{F}$ is either
$\binom{[n]}{k}$, a star $\mathcal{S}_r$, or the complement
$\binom{[n]}{k}\setminus\mathcal{S}_r$ of a star.
\end{enumerate}
\end{theorem}

\begin{proof}
The equality \eqref{eq:innerproduct-n-equals-k} is obvious when $n=k$.
Assume that $n>k$.
For $d=0$, both sides of~\eqref{innerproductABthminq} are
$|\mathcal{A}|\,|\mathcal{B}|$.
For $d=k$, the asserted inequality is
\[
|\mathcal{A}\cap\mathcal{B}|
\leq \sqrt{|\mathcal{A}|\,|\mathcal{B}|},
\]
which is trivial, and equality
holds if and only if at least one family is empty or
$\mathcal{A}=\mathcal{B}$.

Fix $n>k>d>0$, write $\lambda_i=\lambda _d(i)$,
and let $\alpha_1\geq \alpha_2\geq \cdots \geq \alpha_N$ be all eigenvalues of $M_d$ with a corresponding real orthonormal eigenbasis
$\v_1,\v_2, \ldots ,\v_N$, where $\v_1=\1\big/\sqrt{N}$ and $N=\binom{n}{k}$.
Assume that $\1_\A=\sum_{t=1}^Na_t \v_t$
and $\1_\B=\sum_{t=1}^Nb_t \v_t$.
Observe that
\begin{equation}\label{a111exp}
a_1=\big\langle \1_\A,\1\big/\sqrt{N}\big\rangle=|\A| \big/\sqrt{N},
~ \left(\sum_{t=1}^N |a_t|^2\right)-|a_1|^2=|\A|-|\A|^2\big/ N,
\end{equation}
\begin{equation}\label{a222exp}
b_1=\big\langle \1_\B,\1\big/\sqrt{N}\big\rangle=|\B| \big/\sqrt{N},
~{\rm and~} \left(\sum_{t=1}^N |b_t|^2\right)-|b_1|^2=|\B|-|\B|^2\big/ N.
\end{equation}
On the one hand, we have
\begin{equation}\label{fMgexpUppe111}
\1_\A^TM_d \1_\B
=\sum_{j=0}^m \binom{k-j}{d} \left(\1_\A^TA_j\1_\B\right)
=\sum_{U\in \binom{[n]}{d}}
d_{\A}(U) d_{\B}(U)
=\big\langle \mathbf{v}_d(\A),\mathbf{v}_d(\B)\big\rangle.
\end{equation}
Indeed, the LHS of \eqref{fMgexpUppe111} counts the number of triples $(U,A,B)\in \binom {[n]}{d}\times \A \times \B$ such that $U\subseteq A\cap B,$ which is equals to
$\sum_{U\in \binom {[n]}{d}}
d_{\A}(U) d_{\B}(U)$.
On the other hand, we have
\begin{align}\label{ABcrossinq111}
\1_\A^TM_d \1_\B
=& \sum_{t=1}^N a_t b_t \alpha_t
\leq a_1 b_1 \lambda _0 + \lambda _1 \sum_{t=2}^N |a_t b_t |\\ \label{ABcrossinq222}
\leq & a_1 b_1\lambda _0 + \lambda _1
\left(\sum_{t=2}^N |a_t|^2\right)^{\frac{1}{2}}
\left(\sum_{t=2}^N |b_t|^2\right)^{\frac{1}{2}}
&& \text {(by Cauchy--Schwarz)}\\ \nonumber
=& a_1 b_1 \lambda _0 + \lambda _1\left(|\A|-|\A|^2/ N \right)^{\frac{1}{2}}
\left(|\B|-|\B|^2/ N \right)^{\frac{1}{2}}
&& \text {(by  \eqref{a111exp} and \eqref{a222exp})}\\\label{ABcrossinq333}
\leq &
a_1 b_1 \lambda _0 + \lambda _1\left( \big(|\A||\B|\big)^{\frac{1}{2}}- \frac{|\A||\B|}{N} \right)&& \text {(by Lemma \ref{specialAGinq})} \\  \nonumber
=& \frac{|\A||\B|}{N} \big(\lambda _0-\lambda _1\big)+\big(|\A||\B|\big)^{\frac{1}{2}}\lambda _1\\\label{ABcrossinq444}
=&\frac{\binom{k}{d}\binom{k-1}{d}}{\binom{n-1}{d}}|\A||\B|
+\binom{k-1}{d-1} \binom{n-d-1}{k-d}\big(|\A||\B|\big)^{\frac{1}{2}},
&& \text {(by Lemma \ref{eigenvalue01ofMd})}
\end{align}
which implies \eqref{innerproductABthminq} with the aid of \eqref{fMgexpUppe111}, as desired.

Assume that  $|\A||\B|>0$ and
the equality holds in \eqref{innerproductABthminq}.
The equality occurs at \eqref{ABcrossinq333} implies that $|\A|=|\B|$ by Lemma \ref{specialAGinq}, and hence that
$a_1 =b_1 $ and $\sum_{t=2}^N |a_t|^2=\sum_{t=2}^N |b_t|^2$ by  \eqref{a111exp} and \eqref{a222exp}.
For the Cauchy--Schwarz step \eqref{ABcrossinq222} to be an equality,
we must have $|a_t|= \alpha |b_t|$ for some $\alpha$ and all $2\leq t \leq N,$
which implies that $|a_t|= |b_t|$ for all $1\leq t \leq N.$
Moreover,
we also need equality at \eqref{ABcrossinq111},
so
$\sum_{t=2}^N a_t b_t \alpha_t=\lambda _1 \sum_{t=2}^N |a_t b_t |$.
For every $2\leq t\leq N$,
$$\lambda_1|a_tb_t|-\alpha_ta_tb_t$$
is nonnegative.
Hence equality forces $a_tb_t\geq 0$ when
$\alpha_t=\lambda_1$, and $a_tb_t=0$ when $\alpha_t<\lambda_1$.
Recall that $|a_t|= |b_t|$ for all $1\leq t \leq N,$
so we must have
$a_t= b_t$ for all $1\leq t \leq N$ and $a_t=0$ when $\alpha_t< \lambda _1$.
Thus, $\1_\A =\1_\B$ lies in the
the subspace spanned by eigenvectors of eigenvalues $\lambda _0$ and $\lambda _1$, that is $E_0\oplus E_1.$
If $n=k+1$, then $m=1$ and $E_0\oplus E_1=\mathbb R^N$, so every
nonempty common family $\A=\B$ gives equality, as asserted in~(i).
Assume that $n\geq k+2$.
Since $\1_\A\in E_0\oplus E_1$, the Boolean
function $f_\A$ is affine on the slice.  The classification of Boolean
degree-one functions on a slice~\cite[Lemma~3.4]{Filmus} implies that
$f_\A$ is one of $\{1, f_{\S_i},1-f_{\S_i}\}$ for some
$i\in [n]$, which  proves~(ii) and completes the proof.
\end{proof}

To prove the `stability' part of Theorem~\ref{EKR2normmainthm}, we need the following stability result of
Theorem~\ref{innerproductABthm}.

\begin{lemma}\label{stabilitylemmaip}
Let $n>k>d>0$ be integers.
If $\A,\B\subseteq \binom{[n]}{k}$ such that
$$\big\langle\v_d(\A),\v_d(\B)\big\rangle\geq (1-\epsilon)\left(\frac{\binom{k}{d}\binom{k-1}{d}}{\binom{n-1}{d}} X^2
+\binom{k-1}{d-1} \binom{n-d-1}{k-d} X \right)$$
for some $0\leq \epsilon \leq 1$ and a positive number $X$, then
$|\A||\B|\geq (1-2\epsilon)X^2.$
\end{lemma}

\begin{proof}
With the same notation in the proof of Theorem \ref{innerproductABthm}, set $D_\A^2:=\sum_{t=n+1}^N |a_t|^2$ and $D_\B^2:=\sum_{t=n+1}^N |b_t|^2$.
We first treat the case $n\geq k+2$.
In this case, $m=\min\{k,n-k\}\geq 2$, so $\lambda_2$ is well-defined.
One may check that
\begin{align*}
\big\langle \mathbf{v}_d(\A),\mathbf{v}_d(\B)\big\rangle
=&\1_\A^TM_d \1_\B
=   \sum_{t=1}^N a_t b_t \alpha_t
\leq    a_1 b_1 \lambda _0 + \lambda _1 \sum_{t=2}^n |a_t b_t |+ \lambda _2 \sum_{t=n+1}^N |a_t b_t |\\
\leq & a_1 b_1 \lambda _0 + \lambda _1\left(\sum_{t=2}^n |a_t|^2\right)^{\frac{1}{2}}
\left(\sum_{t=2}^n |b_t|^2\right)^{\frac{1}{2}}
+ \lambda _2\left(\sum_{t=n+1}^N |a_t|^2\right)^{\frac{1}{2}}
\left(\sum_{t=n+1}^N |b_t|^2\right)^{\frac{1}{2}}\\
= &a_1 b_1 \lambda _0 + \lambda _1\Big(|\A|-|\A|^2/ N -D_\A^2\Big)^{\frac{1}{2}}
\Big(|\B|-|\B|^2/ N -D_\B^2\Big)^{\frac{1}{2}}
+ \lambda _2D_\A D_\B\\
\overset{(\ast)}{\leq} & a_1 b_1 \lambda _0 + \lambda _1\left(\big(|\A||\B|\big)^{\frac{1}{2}}- \frac{|\A||\B|}{N}-D_\A D_\B\right)
+ \lambda _2D_\A D_\B\\
= & \frac{|\A||\B|}{N}\left(\lambda _0 - \lambda _1\right)+\lambda _1\big(|\A||\B|\big)^{\frac{1}{2}}
+\left(\lambda _2-\lambda _1\right)D_\A D_\B.
\end{align*}
where the inequality ($\ast$) is obtained by using Lemma \ref{specialAGinq} twice.
On the other hand, by the assumption and Lemma \ref{eigenvalue01ofMd}, we have
\begin{align*}
\big\langle\v_d(\A),\v_d(\B)\big\rangle
\geq& (1-\epsilon) \left(\frac{\binom{k}{d}\binom{k-1}{d}}{\binom{n-1}{d}} X^2
+\binom{k-1}{d-1} \binom{n-d-1}{k-d} X\right)\\
=&(1-\epsilon)\left( \frac{\lambda _0-\lambda _1}{ \binom{n}{k} } X^2
+\lambda _1  X\right).
\end{align*}
Combining these inequalities, we get that
\begin{align*}
\left(\lambda _2-\lambda _1\right)D_\A D_\B
\geq& \frac{\lambda _0-\lambda _1}{ \binom{n}{k} }
\left((1-\epsilon) X^2- |\A||\B|\right)
+\lambda _1 \left((1-\epsilon) X-\big(|\A||\B|\big)^{\frac{1}{2}}\right)\\
\geq& \frac{\lambda _0-\lambda _1}{ \binom{n}{k} }
\left((1-\epsilon)^2 X^2- |\A||\B|\right)
+\lambda _1 \left((1-\epsilon) X-\big(|\A||\B|\big)^{\frac{1}{2}}\right).
\end{align*}
Note that $\left(\lambda _2-\lambda _1\right)D_\A D_\B$ is non-positive,
so we deduce that
$$|\A||\B|\geq (1-\epsilon)^2X^2
\geq(1-2\epsilon) X^2,$$
as desired.

It reminds to consider the case of $n=k+1$.
In this case, $m=\min\{k,n-k\}=1$ and $N=\binom{n}{k}=n$, so the matrix $M_d$ has only the two eigenvalues $\lambda_0=\lambda_d(0)$ and $\lambda_1=\lambda_d(1)$.
By the same discussion of $n\geq k+1$, we have
$$
(1-\epsilon)\left( \frac{\lambda _0-\lambda _1}{ \binom{n}{k} } X^2
+\lambda _1  X\right)
\leq
\big\langle \mathbf{v}_d(\A),\mathbf{v}_d(\B)\big\rangle
\leq  \frac{\lambda_0-\lambda_1}{\binom{n}{k}} {|\mathcal A||\mathcal B|} +\lambda_1 \big(|\A||\B|\big)^{\frac{1}{2}}.
$$
Thus,
$$\frac{\lambda _0-\lambda _1}{ \binom{n}{k} }
\left((1-\epsilon)^2 X^2- |\A||\B|\right)
+\lambda _1 \left((1-\epsilon) X-\big(|\A||\B|\big)^{\frac{1}{2}}\right)\leq 0.$$
So we deduce that
$$|\A||\B|\geq (1-\epsilon)^2X^2
\geq(1-2\epsilon) X^2,$$
which proves the result.
\end{proof}

We are now ready to prove the main results of this paper.

\begin{proof}[{\bf Proof of Theorem~\ref{EKR2normmainthm}}]
If $k\leq n<2k$,
then any two family $\A,\B\subseteq \binom{[n]}{k}$ are cross-intersecting.
For every
$U\in\binom{[n]}d$, we have
\[
d_{\mathcal A}(U),d_{\mathcal B}(U)
\leq\binom{n-d}{k-d},
\]
which gives
\[
\big\langle\mathbf v_d(\mathcal A),\mathbf v_d(\mathcal B)\big\rangle
\leq\binom nd\binom{n-d}{k-d}^{2}=\Gamma(n,k,d).
\]
For $d=0$, equality forces $|\mathcal A|=|\mathcal B|=\binom nk$.
For $d>0$, equality forces every $d$-set to have its maximum possible degree in both families, which again gives
$\mathcal A=\mathcal B=\binom{[n]}k$.

Assume that $n\geq2k$.
We next completes the proofs of (a) and (b) by considering the following three cases.

{\bf Case 1}: $d=0$.
In this case, we have
\[
\big\langle\mathbf{v}_0(\mathcal{A}),
\mathbf{v}_0(\mathcal{B})\big\rangle
=|\mathcal{A}|\,|\mathcal{B}|
\qquad\text{and}\qquad
\gamma(n,k,0)=\binom{n-1}{k-1}^2.
\]
For $n>2k$, the assertions follow from Theorems~\ref{EKRcross} and
\ref{EKRcrossstabilitythm}.
For $n=2k$, as $\A,\B\subseteq \binom{[n]}{k}$ are cross-intersecting, we deduce that $\mathcal B\subseteq\binom{[n]}k\setminus\overline{\mathcal A}$, where $\overline{\mathcal A}=\{[n]\setminus A:A\in\mathcal A\}$.
Hence,
\[
|\mathcal A||\mathcal B|\leq |\A|(\binom{n}{k}-|\A|)
\leq \binom{n}{k}^2/4=\binom{n-1}{k-1}^2.
\]
where the last equality holds as $n=2k$.
Moreover, equality holds if and only if $|\mathcal A|=|\mathcal B|=\binom{n}{k}/2$ and
$\mathcal B=\binom{[n]}k\setminus\overline{\mathcal A}$, which is
the asserted classification for $d=0$.

{\bf Case 2}: $d=k$.
In this case, we have
\[
\big\langle\mathbf{v}_k(\mathcal{A}),
\mathbf{v}_k(\mathcal{B})\big\rangle
=|\mathcal{A}\cap\mathcal{B}|
\qquad\text{and}\qquad
\gamma(n,k,d)=\binom{n-1}{k-1}.
\]
Since $\mathcal{A}$ and $\mathcal{B}$ are cross-intersecting,
$\mathcal{A}\cap\mathcal{B}$ is intersecting, and the upper bound $\big\langle\mathbf{v}_{d}(\A),\mathbf{v}_{d}(\B)\big\rangle
\leq \gamma(n,k,d)$
follows from Theorem~\ref{EKRcross}.
Assume that the equality holds, that is,
$$\big\langle\mathbf{v}_{d}(\A),\mathbf{v}_{d}(\B)\big\rangle
=|\mathcal{A}\cap\mathcal{B}|= \gamma(n,k,d)=\binom{n-1}{k-1}.$$
For $n>2k$, we have
\[
|\mathcal{A}|\,|\mathcal{B}|
\geq|\mathcal{A}\cap\mathcal{B}|^2
=\binom{n-1}{k-1}^2.
\]
The equality statement in Theorem~\ref{EKRcross} implies that
$\mathcal{A}=\mathcal{B}$ is a star.
For $n=2k$, we have that
$\mathcal A\cap\mathcal B=:\mathcal F$ is an intersecting family of
size $\binom{n-1}{k-1}= \frac{1}{2}\binom{n}{k}$.
Since the members of $\binom{[n]}{k}$ occur in $\binom{n-1}{k-1}$ complementary pairs
$
 \{C,[n]\setminus C\},
$
an intersecting family contains at most one member of each pair.
Thus, $\F$ contains exactly one member of every complementary pair.
We claim that $\A=\B=\F$.  Indeed, if
$H\in\A\setminus\F$, then $H\notin\F$, and hence
$[n]\setminus H\in\F\subseteq\B$.  This contradicts the
cross-intersection of $\A$ and $\B$.  Thus $\A=\F$, and the same
argument gives $\B=\F$.
Conversely, if $\F$ contains exactly one member of every
complementary pair, then $\F$ is intersecting.
Hence
$\A=\B=\F$ is cross-intersecting and satisfies
\[
 \big\langle\v_k(\A),\v_k(\B)\big\rangle
 =|\F|=\binom{n-1}{k-1},
\]
as desired.

{\bf Case 3}: $0<d<k$.
In this case,
$\big\langle\mathbf{v}_d(\mathcal{A}),
\mathbf{v}_d(\mathcal{B})\big\rangle
\leq \gamma(n,k,d)$ follows from Theorem~\ref{innerproductABthm} and Theorem~\ref{EKRcross}.
If $n>2k$, equality throughout forces
$|\mathcal{A}|\,|\mathcal{B}|=\binom{n-1}{k-1}^2$, and
Theorem~\ref{EKRcross} yields $\mathcal{A}=\mathcal{B}$ is a star.
Conversely, a common star attains the equality.
If $n=2k$, equality implies
$|\mathcal A||\mathcal B|=\binom{n-1}{k-1}^2$.
Moreover, equality in Theorem~\ref{innerproductABthm} then forces
$\mathcal A=\mathcal B$ to be the complete family, a star, or the
complement of a star.
The complete family is not intersecting, whereas both of the other two families are intersecting when $n=2k$.
This proves the boundary classification for $0<d<k$.

It remaids to prove (c).
Let $C_1$ be the constant in Theorem \ref{EKRcrossstabilitythm} and set
$C=2C_1$ which is the required constant.
Suppose that
\[
\big\langle\mathbf{v}_d(\mathcal{A}),
\mathbf{v}_d(\mathcal{B})\big\rangle
\geq(1-\epsilon)\gamma(n,k,d).
\]
If $d=0$, then
\[
  \bigl\langle \mathbf{v}_0(\mathcal A),\mathbf{v}_0(\mathcal B)\bigr\rangle
  =|\mathcal A||\mathcal B|
  \qquad\text{and}\qquad
  \gamma(n,k,0)=\binom{n-1}{k-1}^2.
\]
The assertion follows from Theorem~\ref{EKRcrossstabilitythm}.
If $d=k$, then
\[
  \bigl\langle \mathbf{v}_k(\mathcal A),\mathbf{v}_k(\mathcal B)\bigr\rangle
  =|\mathcal A\cap\mathcal B|
  \qquad\text{and}\qquad
  \gamma(n,k,k)=\binom{n-1}{k-1},
\]
so we get
\[
  |\mathcal A||\mathcal B|
  \geq |\mathcal A\cap\mathcal B|^2
  \geq(1-\epsilon)^2\binom{n-1}{k-1}^2
  \geq(1-2\epsilon)\binom{n-1}{k-1}^2.
\]
If $0<d<k$,
applying Lemma~\ref{stabilitylemmaip} with
$X=\binom{n-1}{k-1}$, we get
\[
|\mathcal{A}|\,|\mathcal{B}|
\geq(1-2\epsilon)\binom{n-1}{k-1}^2.
\]
Since
\[
\epsilon\leq\frac{k^2}{C^2n^2}
=\frac{k^2}{4C_1^2n^2},
\]
the result follows from Theorem~\ref{EKRcrossstabilitythm} with the parameter $\epsilon'=2\epsilon$ and $C_1\epsilon'= C_12\epsilon=C\epsilon$.
The proof is completed.
\end{proof}

To prove
Theorem~\ref{clusterEKRlp}, we need the following lemma due~\cite[Lemma 2.6]{Cao}.

\begin{lemma}[\cite{Cao}]\label{convextransferlem}
Let $N$ be a positive integer, and let $L,D,M,a$ be real numbers
satisfying
\[
0\leq L<D\leq M,
\qquad
0<a<N,
\]
and put
\[
c:=\frac{aD+(N-a)L}{M},
\qquad
\xi:=\frac LM,
\qquad
\eta:=c-N\xi.
\]
For $0\leq u\leq M$, define
\[
 B(u):=c\xi u^2+D\eta u.
\]
Fix $m\in[0,M]$, and suppose that $x_1,\ldots,x_N$ are nonnegative
real numbers satisfying
\begin{equation}\label{eq:abstract-envelope-assumptions}
x_i\leq\min\{m,D\},\qquad
\sum_{i=1}^N x_i=cm,\qquad
\sum_{i=1}^N x_i^2\leq  B(m).
\end{equation}
Then, for every real $p\geq2$,
\begin{equation}\label{eq:abstract-envelope}
\sum_{i=1}^N x_i^p
\leq
aD^p+(N-a)L^p.
\end{equation}
If $p>2$, equality forces $m=M$, $a\in\mathbb Z$ and, after
reordering,
\[
(x_1,\ldots,x_N)
=
(\underbrace{D,\ldots,D}_{a},
\underbrace{L,\ldots,L}_{N-a}).
\]
\end{lemma}

\begin{proof}[{\bf Proof of Theorem~\ref{clusterEKRlp}}]
Set
	\[
	m:=|\A|
	\qquad\text{and}\qquad
	M:=\binom{n-1}{k-1}.
	\]
	Theorem~\ref{dclusterEKR}(a) gives $m\leq M$.
Notice also that
	\[
	n\geq \frac{tk}{t-1}=k+\frac{k}{t-1}
	\geq k+\frac{k}{k-1}>k+1.
	\]
We write
\[
\gamma(u):=
	\frac{\binom kd\binom{k-1}d}{\binom{n-1}d}u^2
	+\binom{k-1}{d-1}\binom{n-d-1}{k-d}u
	\qquad (u\geq0).
\]

	

For $d=0$, the degree vector has the single entry $m$,
whereas for $d=k$ it is the characteristic vector of $\A$.
Hence
	\[
	\|\v_0(\A)\|_p^p=m^p,
	\qquad
	\|\v_k(\A)\|_p^p=m.
	\]
On the other hand,
	$\Phi_p(n,k,0)=M^p$ and $\Phi_p(n,k,k)=M$.
The upper bound and the equality assertions for $d\in\{0,k\}$ therefore follow directly from Theorem~\ref{dclusterEKR}(a),(b).

Assume that $1\leq d\leq k-1$.
For $U\in\binom{[n]}d$, write
	\[
	x_U:=d_\A(U).
	\]
We apply Lemma~\ref{convextransferlem} to this degree sequence with the parameters
\begin{equation}\label{eq:cluster-transfer-parameters}
N=\binom nd, \quad a =\binom{n-1}{d-1}, \quad D=\binom{n-d}{k-d},	\quad
L =\binom{n-d-1}{k-d-1}, \text{~ and ~}M =\binom{n-1}{k-1}.
	\end{equation}
	These parameters satisfy
	\[
	0<L<D\leq M,
	\qquad
	0<a<N.
	\]
Indeed, $D/L=(n-d)/(k-d)>1$, and the $D$ members of $\binom{[n]}k$ containing a fixed $d$-set form a subfamily of a point-star.
Moreover,
	\[
	0\leq x_U\leq\min\{m,D\}
	\quad\text{for every }U\in\binom{[n]}d,
	\]
and double counting the pairs $(U,A)$ with $U\subseteq A\in\A$ gives
	\begin{equation}\label{eq:cluster-first-moment}
		\sum_{U\in\binom{[n]}d}x_U=\binom kd m.
	\end{equation}
The elementary identities
	\[
	aD+(N-a)L=\binom kd M,
	\qquad
	\frac LM=\frac{\binom{k-1}d}{\binom{n-1}d}
	\]
show that the parameters $c,\xi,\eta$ in
Lemma~\ref{convextransferlem} satisfy
	\begin{equation}\label{eq:cluster-transfer-identities}
		c=\binom kd,
		\qquad
		c\xi=\frac{\binom kd\binom{k-1}d}{\binom{n-1}d},
		\qquad
		D\eta=\binom{k-1}{d-1}\binom{n-d-1}{k-d}.
	\end{equation}
	Taking $\mathcal B=\mathcal A$ in
	Theorem~\ref{innerproductABthm}, and using
	\eqref{eq:cluster-transfer-identities}, we obtain
	\begin{equation}\label{eq:cluster-second-moment}
		\sum_{U\in\binom{[n]}d}x_U^2
		\leq c\xi m^2+D\eta m= B(m).
	\end{equation}
	Thus every hypothesis of Lemma~\ref{convextransferlem} is satisfied,
	and that lemma yields
	\[
	\sum_{U\in\binom{[n]}d}x_U^p
	\leq aD^p+(N-a)L^p
	=\Phi_p(n,k,d),
	\]
which proves (a). 

We next determine the equality cases.
If $p=2$ and equality holds in \eqref{eq:cluster-lp-bound}, Theorem~\ref{innerproductABthm} and $m\leq M$ by Theorem~\ref{dclusterEKR}(a) imply
$\|\v_d(\A)\|_2^2\leq \gamma(M).$
Equality therefore forces $m=M$ and equality in
Theorem~\ref{innerproductABthm}.
If $(t,n)\neq(2,2k)$,
Theorem~\ref{dclusterEKR}(b) gives that $\A$ is a star.  If
	$(t,n)=(2,2k)$, the equality statement in
Theorem~\ref{innerproductABthm}, together with $n\geq k+2$ and
$m=M$,
gives $\A$ is a star or the complement of a star.
Both families are intersecting and attain $\Phi_2(n,k,d)$.
	
	Now assume that $p>2$ and equality holds in
	\eqref{eq:cluster-lp-bound}.
The equality statement in Lemma~\ref{convextransferlem} forces $m=M$ and forces the multiset of $d$-degrees of $\A$ to be
	\begin{equation}\label{eq:cluster-equality-degrees}
		(\underbrace{D,\ldots,D}_{a},
		\underbrace{L,\ldots,L}_{N-a}).
	\end{equation}
In particular, the equality also holds in Theorem~\ref{innerproductABthm}.
Except the case of $(t,n)=(2,2k)$, Theorem~\ref{dclusterEKR}(b) again gives a star.
For $(t,n)=(2,2k)$,
Theorem~\ref{innerproductABthm} implies $\A$ is a star or the complement of a star.
The latter family has $d$-degree zero at every $d$-set containing the omitted point.
Recall that$L>0$,  so its degree multiset cannot be~\eqref{eq:cluster-equality-degrees} and hence that $\A$ is a star.
Conversely, a star is $t$-cluster-free and satisfies the equality.
This proves~(b).

	
It remains to prove stability.
Fix $\delta>0$.
By Theorem~\ref{dclusterEKR}(c), there are a number $\eta\in(0,1)$ and an integer $n_1$ such that, whenever $n>n_1$ and $m\geq(1-\eta)M,$
	there exists an $S\in\binom{[n]}{n-1}$ such that
	\[
	\left|\A\cap\binom Sk\right|<\delta M.
	\]
	It is therefore enough to show that a sufficiently near-extremal
	$\ell_p$-norm forces $m\geq(1-\eta)M$.
	
For $d=0$, the inequality $m^p\geq(1-\epsilon)M^p$ implies
$m\geq(1-\epsilon)^{1/p}M\geq(1-\epsilon)M$.
For $d=k$, the same conclusion follows immediately from
$m\geq(1-\epsilon)M$.
Thus these cases follow by choosing $\epsilon\leq\eta$.
	
Suppose that $1\leq d\leq k-1$.  With the notation in
	\eqref{eq:cluster-transfer-parameters}, define
	\begin{equation}\label{eq:cluster-rho}
		\rho_n:=
		\frac{\Phi_p(n,k,d)}{D^{p-2}\Phi_2(n,k,d)}
		=\frac{\dfrac dn+\left(1-\dfrac dn\right)
			\left(\dfrac{k-d}{n-d}\right)^p}
		{\dfrac dn+\left(1-\dfrac dn\right)
			\left(\dfrac{k-d}{n-d}\right)^2}.
	\end{equation}
For fixed $k,d,p$, we have $\rho_n\to1$ as $n\to\infty$.
When $p=2$, observe that $\rho_n=1$.
Enlarge $n_1$, if necessary, so that
	\[
	\rho_n\geq1-\frac\eta4
	\qquad\text{whenever }n>n_1,
	\]
	and choose $0<\epsilon\leq\eta/4$.  Since
	$0\leq x_U\leq D$ and $p\geq2$, we have
	$x_U^p\leq D^{p-2}x_U^2$.
Consequently, the near-extremal hypothesis and Theorem~\ref{innerproductABthm} imply
$$
		(1-\epsilon)\Phi_p(n,k,d) \leq\sum_{U\in\binom{[n]}d}x_U^p
\leq D^{p-2}\sum_{U\in\binom{[n]}d}x_U^2
		\leq D^{p-2}\gamma(m).
$$
	Since $\gamma(M)=\Phi_2(n,k,d)$, it follows from~\eqref{eq:cluster-rho}
	that
	\[
	\gamma(m)\geq(1-\epsilon)\rho_n \gamma(M)
	\geq\left(1-\frac\eta4\right)^2\gamma(M)
	\geq\left(1-\frac\eta2\right)\gamma(M).
	\]
	On the other hand,
	\[
	\gamma((1-\eta)M)\leq(1-\eta)\gamma(M)
	<\left(1-\frac\eta2\right)\gamma(M).
	\]
	The strict monotonicity of $\gamma$ therefore gives
	$m>(1-\eta)M$.  The stability assertion now follows from
	Theorem~\ref{dclusterEKR}(c).
\end{proof}

\section{Concluding Remarks}\label{ConclusionSec}

In this paper, we present an Erd\H{o}s--Ko--Rado theorem for cross-intersecting families in the Euclidean inner product along with a corresponding stability result.
This inner product version includes both the classical
Erd\H{o}s--Ko--Rado theorem and its cross-intersecting analogue as special cases.
Our proof relies on the cross version of Bey's inequality, that is, Theorem~\ref{innerproductABthm}.
This inequality transfers suitable cardinality bounds into exact
$\ell_p$-norm bounds when the extremal configuration is a star.
For instance, Chv\'atal~\cite{Chvatal} conjectured that the star is the maximum size $k$-uniform family with no $t$-simplex for all $n\geq (t+1)k/t$.
Frankl and F\"redi~~\cite{FranklF} proved it for $n\geq n_0(k,t)$.
The method proving of Theorem~\ref{clusterEKRlp} can be used to prove a similar result for $t$-simplex-free families.

We conclude this paper with an open problem in~\cite{Cao}.
A family  $\A \subseteq \binom{[n]}{k}$ is called {\em $t$-intersecting} if $|S\cap T|\geq t$ for any $S,T\in \A$.
Erd\H{o}s, Ko, and Rado~\cite{Erdos} proved that
there exists an integer $n_0(k, t)$ such that
if $n\geq n_0(k, t)$ and $\A \subseteq \binom{[n]}{k}$ is $t$-intersecting, then
$|\A| \leq \binom{n-t}{k-t}.$
Wilson~\cite{Wilson} showed that the smallest possible such $n_0(k, t)$ is $(t+1)(k-t+1)$.
Inspired by these classical works,
Brooks and Linz~\cite[Conjecture 4.1]{Brooks} conjectured that
if $n\geq (t+1)(k-t+1)$ and
$\mathcal{A} \subseteq \binom{[n]}{k}$ is $t$-intersecting,
then
$$||\mathbf{v}_{k-1}(\A)||^2\leq ||\mathbf{v}_{k-1}\big(\mathcal{S}_T\big)||^2
=\binom{n-t}{k-t}\big( (k-t)(n-k+1)+t\big),$$
where $\mathcal{S}_T$ is a $t$-star, that is,
$\mathcal{S}_T:=
\left\{S\in\binom{[n]}k: T\subseteq S\right\}$ for some $T\in\binom{[n]}{t}$.
Recently, Wu and Zhang~\cite{Wu} confirmed it and gave many related extremal results.
Inspired by these works, Cao, Lu, and Zhang~\cite{Cao} proposed the following conjecture.

\begin{conjecture}[Conjecture 6.1,~\cite{Cao}]\label{Caoconj}
Let $k\geq t\geq 1$, $k-1\geq d\geq 1$, and let $p\geq2$ be real.
If $n\geq(t+1)(k-t+1)$ and $\A\subseteq\binom{[n]}k$ is $t$-intersecting, then
\[
 \|\mathbf{v}_d(\A)\|_p^p \leq \|\mathbf{v}_d(\mathcal{S}_T)\|_p^p,
\]
where $\mathcal{S}_T$ is a $t$-star.
\end{conjecture}

The range in Conjecture~\ref{Caoconj} is the sharp by the classical work of Wilson~\cite{Wilson}.
Cao, Lu, and Zhang~\cite{Cao} proved this conjecture for the cases of $d=k-1$ and $t=1$, respectively.
We refer the reader to their paper~\cite{Cao} for details.


\subsection*{Acknowledgments}
This work was initiated when the first author was a visiting Ph.D. student at the National University of Singapore
from October 2022 to October 2024,
with the support of the China Scholarship Council (CSC grant no. 202206500005).
The first author would like to thank
his visiting adviser Hao Huang for his guidance and valuable suggestions that contributed to the improvement of the paper.
The authors would like to thank Jian Wang for helpful comments on an earlier draft.


\end{document}